\RequirePackage[l2tabu, orthodox]{nag}
\documentclass[11pt]{amsart} 
\author{Andy Hammerlindl}
\title{Properties of compact center-stable submanifolds}
\keywords{partial hyperbolicity, attractors, center-stable submanifolds}
\subjclass{37D30, 37C70}

\usepackage{amsmath}
\usepackage{amssymb}
\usepackage{amsfonts}
\usepackage{amsthm}
\usepackage{amscd}
\usepackage{graphicx}
\usepackage{setspace}
\usepackage{cleveref}
\usepackage{fourier}
\usepackage[T1]{fontenc}

\makeatletter
\def\saveenum{\xdef\@savedenum{\the\c@enumi\relax}}
\def\resetenum{\global\c@enumi\@savedenum}
\makeatother

\newcommand{\HHU}{Rodriguez Hertz, Rodriguez Hertz, and Ures}
\newcommand{\ccss}{compact $cs$-submanifold}
\newcommand{\ccsss}{compact $cs$-submanifolds}

\newcommand{\bbR}{\mathbb{R}}

\newcommand{\bbZ}{\mathbb{Z}}

\newcommand{\bbT}{\mathbb{T}}

\newcommand{\Es}{E^s}
\newcommand{\Ec}{E^c}
\newcommand{\Eu}{E^u}
\newcommand{\Ecu}{E^{cu}}
\newcommand{\Ecs}{E^{cs}}

\newcommand{\Ws}{W^s}

\newcommand{\Wu}{W^u}
\newcommand{\Wcu}{W^{cu}}

\newcommand{\inv}{^{-1}}

\newcommand{\invn}{^{-n}}

\newcommand{\del}{\partial}

\newcommand{\without}{\setminus}

\newcommand{\ep}{\epsilon}
\newcommand{\lam}{\lambda}
\newcommand{\Lam}{\Lambda}

\newcommand{\sig}{\sigma}

\newcommand{\al}{\alpha}

\newcommand{\bt}{\beta}

\newcommand{\qandq}{\quad \text{and} \quad}

\newcommand{\dist}{\operatorname{dist}}

\newcommand{\id}{\operatorname{id}}

\newcommand{\pis}{\pi^s}
\newcommand{\piu}{\pi^u}

\newcommand{\Hu}{H^u}
\newcommand{\Hs}{H^s}

\newcommand{\vs}{v^s}
\newcommand{\vcu}{v^{c u}}

\newcommand{\Cone}{\mathcal{C}}
\newcommand{\Cstar}{\Cone^*}

\newcommand{\subof}{\subset}
\newcommand{\sans}{\setminus}
\newcommand{\ti}{\times}

\newcommand{\Dpsi}{D_{\psi}}
\newcommand{\half}{\tfrac{1}{2}}
\newcommand{\phalf}{p \inv(\tfrac{1}{2})}

\numberwithin{equation}{section}
\newtheorem{thm}[equation]{Theorem}
\newtheorem{cor}[equation]{Corollary}
\newtheorem{lemma}[equation]{Lemma}
\newtheorem{prop}[equation]{Proposition}
\newtheorem{question}[equation]{\textbf{Question}}

\theoremstyle{remark}

\newtheorem*{notation} {\textbf{Notation}}

\providecommand{\acknowledgement}{{\noindent \textbf{Acknowledgements}}\quad}

\begin{document}

\maketitle

\begin{abstract}
    We show that a partially hyperbolic system can have at most a
    finite number of compact center-stable submanifolds.
    We also give sufficient conditions for these submanifolds to exist
    and consider the question of whether they can intersect each other.
\end{abstract}
\section{Introduction} 

Much of the early theory of partially hyperbolic dynamics
was developed by
Hirsch, Pugh, and Shub
in their foundational text,
\emph{Invariant Manifolds} \cite{HPS}.
The book first considers the case of an invariant compact submanifold
of the phase space where the dynamics normal to the submanifold
is hyperbolic.
Later chapters deal with systems where a partially hyperbolic splitting holds
on the entire phase space.
These two cases may overlap.
For instance, if the system has a global splitting of the form
$TM = \Eu \oplus \Ec \oplus \Es$,
it may also have a compact submanifold tangent to $\Ec$
and such a submanifold is therefore normally hyperbolic.

Recent discoveries show that a slightly different possibility exists.
\HHU{} constructed an example of a partially hyperbolic system on the 3-torus
with a compact submanifold, a 2-torus, tangent to the center and stable
directions, $\Ec \oplus \Es$ \cite{rhrhu2016coherent}.
This center-stable submanifold is transverse to the expanding unstable
direction and is therefore a normally repelling submanifold.
Based on this, the author constructed further examples
of compact center-stable submanifolds, both in dimension 3
and higher \cite{ham20XXconstructing}.

The paper establishes general properties for these
types of dynamical objects.
In particular, we show that any partially hyperbolic system
may have at most finitely many compact center-stable submanifolds
and we give sufficient conditions under which these objects exist.
Finally, we consider the consider the question of whether two of these
submanifolds can have non-empty intersection.

\section{Statement of results} 

A diffeomorphism $f$ of a closed connected manifold $M$
is \emph{(strongly) partially hyperbolic}
if there is a splitting of the tangent bundle
\[
    TM = \Es \oplus \Ec \oplus \Eu
\]
such that each subbundle is non-zero and
invariant under the derivative $Df$
and
\[
    \|Df v^s\| < \|Df v^c\| < \|Df v^u\|
    \qandq
    \|Df v^s\| < 1 < \|Df v^u\|
      \]
hold for all $x  \in  M$ and unit vectors
$v^s  \in  \Es(x)$, 
$v^c  \in  \Ec(x)$, and
$v^u  \in  \Eu(x)$.
There exist unqiue foliations $\Ws$ and $\Wu$ tangent to $\Es$ and $\Eu$.
An immersed submanifold $S \subof M$
is a \emph{center-stable submanifold} if it is tangent
to $\Ecs = \Ec \oplus \Es$.

\begin{thm} \label{thm:finite}
    A partially hyperbolic diffeomorphism
    has at most a finite number of
    compact center-stable submanifolds.
\end{thm}
From this, the following result could be proved.

\begin{thm} \label{thm:periodic}
    Every compact center-stable submanifold is periodic.
\end{thm}
However,
in \cref{sec:periodic}
we actually establish \cref{thm:periodic} first and then use the result
to show \cref{thm:finite}.

\medskip{}

While being tangent to $\Ecs$ clearly requires the submanifold to be at least $C^1$,
it is equivalent to a condition which may be stated for $C^0$ submanifolds.

\begin{thm} \label{thm:fromzero}
    Suppose $f:M \to M$ is partially hyperbolic
    and $\Lam \subset M$ is a periodic compact $C^0$ submanifold.
    Then $\Lam$ is a $C^1$ submanifold tangent to $\Ecs$
    if and only if $\Wu(x) \cap \Lam = \{x\}$
    for all $x \in \Lam$.
\end{thm}
This also gives a way to find periodic submanifolds from non-periodic ones.

\begin{thm} \label{thm:toperiodic}
    Suppose $f:M \to M$ is partially hyperbolic,
    $k  \ge  1$,
    and $S \subset M$ is a compact $C^0$ submanifold such that
    \begin{math}
        \Wu(x) \cap S = \{x\}
      \end{math}
    and 
    \begin{math}
        \Wu(x) \cap f^k(S)  \ne  \varnothing
      \end{math}
    for all $x \in S$.
    Then
    there exists a compact center-stable submanifold.
\end{thm}
Theorems \ref{thm:fromzero} and \ref{thm:toperiodic} are proved in
\cref{sec:fromzero}.

The next theorem assumes a one-dimensional unstable direction.
It basically states that if a region $M_0 \subof M$
has two boundary components and the ends of unstable curves inside this region
tend towards the boundary in a uniform way,
then there must be a 
compact center-stable submanifold inside the region.

\begin{thm} \label{thm:phcross}
    Let $f$ be a partially hyperbolic diffeomorphism
    of a manifold $M$,
    $M_0$ a compact connected submanifold of $M$ with boundary,
    $g:M_0 \to [0,1]$ a continuous function,
    and $\ell > 0$ such that

    \begin{enumerate}
        \item $\dim \Eu = 1$,

        \item $\dim M = \dim M_0$,

        \item $f(M_0) = M_0$,

        \item if $x \in \del M_0$, then $\Wu(x) \subset \del M_0$,

        \item $g(\del M_0) = \{0,1\}$,

        \item if $x \in M_0$, $0 < g(x) < 1$, $y \in \Wu(x)$, and $d_u(x,y) > \ell$,
        then
        $g(y) \in \{0,1\}$, and

        \item if $\al : \bbR  \to  M_0$ is a parameterized unstable leaf,
        then
        \[            \lim_{t \to +\infty} g \al(t) = 
            \lim_{t \to +\infty} g f \al(t)
            \qandq
            \lim_{t \to -\infty} g \al(t) = 
            \lim_{t \to -\infty} g f \al(t).
        \]  \end{enumerate}
    Then, there is a compact center-stable submanifold in the
    interior of $M_0$.
\end{thm}
This result will be used in an upcoming paper
as a critical step in giving a classification of all partially hyperbolic
systems in dimension three which have center-stable tori.
\Cref{sec:cross} gives the proof of \cref{thm:phcross}.

\medskip

The proofs of the above results
never use the sub-splitting $\Ecs = \Ec \oplus \Es$ of the center-stable
bundle.
Therefore, all of the above results also hold for \emph{weakly}
partially hyperbolic systems, where the diffeomorphism 
has 
an invariant splitting of the form  $TM = \Ecs \oplus \Eu$.
For further discussion of weak versus strong
partial hyperbolicity,
see sections 1 and 6 of \cite{hp20XXsurvey} and the references therein.

\medskip

Note that \cref{thm:finite} above only shows finiteness;
it does not say anything about disjointedness.

\begin{question}
    Can two distinct compact center-stable submanifolds
    have non-empty intersection?
\end{question}
In the case of strongly partially hyperbolic systems in dimension 3,
we have a number of special tools at our disposal
including branching foliations and Anosov tori
\cite{BBI2, rhrhu2011tori},
and we can answer this question in the negative.

\begin{thm} \label{thm:nocalzones}
    In a 3-dimensional strongly
    partially hyperbolic system,
    the compact center-stable submanifolds
    are pairwise disjoint.
\end{thm}
To suggest why such intersections might be possible
in higher dimensions,
we give an example of an invariant partially hyperbolic subset of
a 3-manifold
which consists of two surfaces glued together, each of which
is tangent to
the center-stable direction of the splitting.
Moreover, the partially hyperbolic splitting on the subset
extends to a dominated splitting defined on the entire manifold.

\begin{thm} \label{thm:calzone}
    There is a diffeomorphism $f : \bbT^3 \to \bbT^3$
    with a sink $z  \in  \bbT^3$ and a global dominated splitting
    into three subbundles
    \begin{math}
        TM = \Eu \oplus \Ec \oplus \Es
    \end{math}
    such that if $B(z)$ denotes the basin of attraction of $z$,
    then the boundary of $B(z)$ is the union of two
    distinct intersecting tori tangent to $\Ec \oplus \Es$,
    and
    the splitting is partially hyperbolic on
    all of\, $\bbT^3 \sans B(z)$.
\end{thm}
This shows in particular that there is no local obstruction
to having an intersection.
We first construct the example which demonstrates \cref{thm:calzone} in \cref{sec:calzone}
and then prove \cref{thm:nocalzones} in \cref{sec:nocalzone}.

\medskip

The above results are stated for center-stable submanifolds.
By replacing $f$ with its inverse,
one may
state analogous results for compact center-\emph{unstable} submanifolds.
It is easier, in some cases, to prove a result in this alternate setting
and so we switch
back and forth
between the two viewpoints in the proofs below.

\bigskip{}

In related work,
theorems \ref{thm:finite} and \ref{thm:periodic} generalize results
given in \cite{rhrhu20XXcsfoln} and their proofs
are based on the techniques given there.
\Cref{thm:fromzero} is closely related to the main result of \cite{bc2016center},
which considers an arbitrary compact invariant set $K$ where
\begin{math}
    \Wu(x) \cap K = \{x\}
\end{math}
for all $x  \in  K$.
\Cref{thm:fromzero} could be proved as a consequence of this result.
However, the fact that $\Lam$ has the structure of a $C^0$ submanifold means that
we can give a direct, intuitive, and
comparatively simple proof of \cref{thm:fromzero} in the space of a few pages.
For this reason, we give a full self-contained proof in this paper.

\section{Finiteness} \label{sec:periodic} 

In this section,
assume $f:M \to M$ is partially hyperbolic.
To prove theorems \ref{thm:finite} and \ref{thm:periodic},
we may freely
replace $f$ by an iterate and therefore also assume
that $\|Df(v)\| > 2\|v\|$ for all non-zero $v \in \Eu$.


Let $d_H$ denote Hausdorff distance.
Equipped with $d_H$, the space
of compact subsets of $M$ is a compact metric space.
If $x$ and $y$ are points on the same unstable leaf, let $d_u(x,y)$
denote the distance between them as measured along the leaf.
If $x$ and $y$ are on distinct unstable leaves,
then $d_u(x,y) = +\infty$.
Similar to the definition of Hausdorff distance,
for subsets $X,Y \subset M$ define
\[
    \dist_u(x, Y) = \inf_{y \in Y} d_u(x, y)
\]
and
\[
    \dist_u(X, Y) =
        \max \big\{
            \sup_{x \in X} \dist_u(x,Y),\,
            \sup_{y \in Y} \dist_u(y,X) \big\}.
\]
In what follows, we write $cs$-submanifold as shorthand for
a center-stable submanifold.
Using the transversality of $\Ecs$ and $\Eu$,
one may prove the following

\begin{lemma} \label{lemma:rhalf}
    There is $r > 0$ such that 
    if $S$ and $T$ are \ccsss{}
    and $d_H(S,T) < r$,
    then $\dist_u(S,T) < \tfrac{1}{2}$.\qed
\end{lemma}
In this section,
call a compact subset $X \subset M$ ``well positioned'' if $d_u(x,y) > 5$
for all distinct $x,y \in X$.

\begin{lemma} \label{lemma:wellpos}
    Let $T$ be a \ccss{}.  Then, there is an integer $n_0$ such that
    $f^n(T)$ is well positioned for all $n > n_0$.
\end{lemma}
\begin{proof}
    As $\Ecs$ is transverse to $\Eu$, there is $\ep > 0$
    such that
    $d_u(x,y) > \ep$
    for all distinct $x \in T$ and $y \in \Ws(x) \cap T$.
    Then take $n_0$ such that $2^{n_0} \ep > 5$ and use the above
    assumption on $\Eu$.
\end{proof}

\begin{lemma} \label{lemma:fibercontract}
    If $S$ is a well-positioned \ccss{} and
    \begin{math}
        \dist_u(S,f^k(S)) \linebreak[3] < \tfrac{1}{2}
    \end{math}
    for some $k  \ge  1$,
    then
    there is a unique well-positioned periodic $C^0$ submanifold $\Lam$
    such that $\dist_u(S, \Lam) < 1$.
\end{lemma}
\begin{proof}
    For $x \in f^k(S)$, define $h(x)$ as the unique point in $f^k(S)$
    such that
    \[    
        d_u(f^{-k}(x),h(x)) < \tfrac{1}{2}.
    \]
    Existence follows from $\dist_u(S, f^k(S)) < \tfrac{1}{2}$
    and uniqueness from the fact that $f^k(S)$ is well positioned.
    By the same reasoning, there an inverse map $h \inv$
    and so
    $h: f^k(S) \to f^k(S)$ is a homeomorphism.
    For $x \in M$, define $\Wu_1(x) = \{ y \in M: \dist_u(x,y) < 1 \}.$
    One can show that
    $f^{-k}(\Wu_1(x)) \subset \Wu_1(h(x))$
    for all $x \in f^{-k}(S)$.
    In other words, $f^{-k}$ restricted to a neighbourhood of $f^k(S)$
    is a fiber contraction of a $C^0$ fiber bundle.
    By the fiber contraction theorem \cite[Theorem 3.1]{HPS},
    there is an $f^{-k}$\,invariant $C^0$ submanifold $\Lam$
    in this neighbourhood.  Applying $f^{-k}$, one sees that
    $\dist_u(S, \Lam) < 2^{-k} < 1$.

    Suppose $\Lam'$ is a well-positioned periodic submanifold
    with $\dist_u(S, \Lam') < 1$.
    Then $\dist_u(\Lam, \Lam') < 2$
    and $\dist_u(f^{-n}(\Lam), f^{-n}(\Lam'))$ tends to zero as $n \to \infty$.
    This shows that $\Lam = \Lam'$.
\end{proof}
\begin{lemma} \label{lemma:closelam}
    Let $S$ be a \ccss{}.
    For any $\ep > 0$, there is
    a well-positioned periodic $C^0$ submanifold $\Lam_\ep$
    such that $\dist_u(S, \Lam_\ep) < \ep$.
\end{lemma}
\begin{proof}
    Let $n_0$ be such that $2^{-n} < \ep$
    and $f^n(S)$ is well positioned for all $n > n_0$.
    As Hausdorff distance defines a compact metric space,
    there are $m > n > n_0$ such that $d_H(f^n(S), f^m(S)) < r$.
    By \cref{lemma:rhalf} and \cref{lemma:fibercontract},
    there is $\Lam$ such that $\dist_u(f^n(S), \Lam) < 1$.
    Then
    $\dist_u(S, f^{-n}(\Lam)) < 2^{-n}$,
    so take
    $\Lam_\ep = f^{-n}(\Lam)$.
\end{proof}
\begin{proof}
    [Proof of \cref{thm:periodic}]
    Let $S$ be a \ccss{}.
    By \cref{lemma:closelam},
    there is a sequence of periodic submanifolds $\{\Lam_n\}$ such that
    $\dist_u(S, \Lam_n)$ tends to zero.
    The uniqueness in \cref{lemma:fibercontract} implies
    that $\{\Lam_n\}$ is eventually constant.
    Therefore, $S = \Lam_n$ for all large $n$.
\end{proof}
\begin{proof}
    [Proof of \cref{thm:finite}]
    All \ccsss{} are periodic.
    By \cref{lemma:wellpos},
    they are all well positioned.
    If $S  \ne  T$ are two of these submanifolds,
    then lemmas \ref{lemma:rhalf} and \ref{lemma:fibercontract}
    imply that
    $d_H(S,T) > r$.
    A compactness argument using Hausdorff distance
    implies that there are only finitely many.
\end{proof}

\section{Regularity of submanifolds} \label{sec:fromzero} 

This section proves theorems \ref{thm:fromzero} and \ref{thm:toperiodic}.
Using the results of the previous section,
the latter follows easily from the former.

%

\begin{proof}
    [Proof of \cref{thm:toperiodic}]
    Note that $\dist_u(S, f^k(S)) < \infty$
    and therefore
    \[    
        \dist_u(f^{-n}(S), f^{k-n}(S)) < \tfrac{1}{2}
    \]
    for sufficiently large $n$.
    The condition $\Wu(x) \cap S = \{x\}$ implies that
    $f^{-n}(S)$ is well positioned.
    \Cref{lemma:fibercontract} then shows that 
    that there is a periodic $C^0$ submanifold
    $\Lam$ which satisfies the hypotheses of \cref{thm:fromzero}.
\end{proof}
The above proof further shows that the submanifolds
in \cref{thm:toperiodic} satisfy $\dist_u(S$, $\Lam) < \infty$.

\medskip{}

One direction of \cref{thm:fromzero} readily follows from
results in the last section.
To prove the other direction, it will be easier to exchange the roles of $\Eu$
and $\Es$.
Therefore, we assume $f:M \to M$ is partially hyperbolic
and $\Lam$ is a periodic $C^0$ submanifold such that
\begin{math}
    \Ws(x) \cap \Lam = \{x\}
\end{math}
for all $x \in \Lam$. Our goal is then to show that $\Lam$ is a $C^1$
submanifold tangent to $\Ecu$.
To prove this, we may freely replace $f$ by an iterate.
In particular, assume $f(\Lam) = \Lam$.
Also assume that
associated to the partially hyperbolic splitting $TM = \Es \oplus \Ecu$
is a continuous function $\lam: M \to (0,\tfrac{1}{2})$
such that
\begin{math}
    \| Df v^s \| < \lam(x) < 2 \lam(x) < \| Df \vcu \|
\end{math}
for all $x \in M$ and unit vectors $\vs \in \Es(x)$ and $\vcu \in \Ecu(x)$.

Let $\Cone \subset TM$ be a cone family associated to the dominated splitting.
That is, for every $x \in M$,
$\Cone(x) = \Cone \cap T_x M$ is a closed convex
subset of $TM$ such that $\Es(x) \subset \Cone(x)$,
$\Ecu(x) \cap \Cone(x) = 0$, and $\Cone(x)$ depends continuously on $x$.
Define the dual cone family
$\Cstar$
as the closure of $TM \sans \Cone$.
The properties of the splitting imply that
\[
    \bigcap_{n  \ge  0} Df^{-n}(\Cone) = \Es
    \qandq
    \bigcap_{n  \ge  0} Df^n(\Cstar) = \Ecu,
\]
Replacing $\Cone$ by some $Df^{-n}(\Cone)$, $f$ by a large iterate $f^m$,
and the function $\lam$ by
\[
    x \mapsto \lam(f^{m-1}(x)) \, \cdots \, \lam(f(x)) \, \lam(x),
\]
assume for any $x \in M$ and non-zero vector $v \in T_x M$ that
\begin{enumerate}
    \item if $v \in \Cone$, then $\| Df v \| < \lam(x) \|v\|$;

    \item if $v  \in  \Cstar$, then $Df(v)  \in  \Cstar$; and

    \item if $v  \in  Df(\Cstar)$, then $\| Df v \| > 2 \lam(x) \|v\|.$
\end{enumerate}
Let $\exp_x : T_x M \to M$ be the exponential map.
Up to rescaling the Riemannian metric on $M$,
assume that if $d(x,y)<1$,
then there is a unique vector $v \in T_x M$ with $\|v\|<1$
such that $y = \exp_x(v)$.
Define a continuous map
\[
    F: \{v \in TM : \|v\| < 1\} \to TM
\]
by requiring that $\exp_{f(x)} (F(v)) = f(\exp_x(v))$
for all $x \in M$ and $v \in T_x M$ with $\|v\|<1$.

\begin{lemma} \label{lemma:Fdelta}
    There is $0 < \delta < 1$ such that
    for any $x \in M$ and $v \in T_x M$ with $\|v\| < \delta$
    \begin{enumerate}
        \item if $v \in \Cone$, then $\| F(v) \| < \lam(x) \|v\|$;

        \item if $v  \in  \Cstar$, then $F(v)  \in  \Cstar$; and

        \item if $v  \in  F(\Cstar)$, then $\| F(v) \| > 2 \lam(x) \|v\|.$
    \end{enumerate}  \end{lemma}
\begin{proof}
    The properties of the exponential map imply that
    \begin{math}
        \tfrac{1}{\|v\|} \|F(v) - Df(v)\|
    \end{math}
    tends uniformly to zero as $v \to 0$.
    The lemma may then be proved from the corresponding properties of
    $Df$.
\end{proof}
Later on, we will also need the following fact.

\begin{lemma} \label{lemma:coneback}
    For any $n > 0$,
    there is $r_n > 0$
    such that if $w  \in  \Cstar$ and $\|w\| < r_n$,
    then
    $F^{n+1}(w)  \in  Df^n(\Cstar)$.
\end{lemma}
\begin{proof}
    There is a lower bound on the angle between any
    non-zero vectors $u  \in  Df^n(\Cone)$ and $v  \in  Df^{n+1}(\Cstar)$.
    Take $r_n$ small enough that if $\|w\| < r_n$,
    then the angle between $F^{n+1}(w)$ and $Df^{n+1}(w)$
    is smaller than this bound.
\end{proof}

\begin{lemma} \label{lemma:Fstable}
    If $v \in T_x M$ is such that $\|F^n(v)\| < \delta$ and $F^n(v) \in \Cone$
    for all $n  \ge  0$, then $\exp_x(v)$ lies on the stable leaf through $x$.
\end{lemma}
This lemma is more or less one of the steps in establishing the existence of
the stable foliation.
See for instance \cite[Section 5]{HPS}.
For completeness,
we give a proof which assumes that the stable foliation exists.

\begin{proof}
    Since $\Ecu$ is transverse to $\Es$, there is $r>0$
    such that the (incomplete) submanifold
    \[
        \Sigma_p := \exp_p \{u \in \Ecu(p) : \|u\| < r \}
    \]
    is transverse to $\Ws$.
    There is also $\eta > 0$ such if $d(p,q) < \eta$,
    then $\Ws(q)$ intersects $\Sigma_p$ in a point $z$
    which satisfies $d(f^n(q), f^n(z)) < \delta/2$ for all $n  \ge  0$.

    Write $x_n := f^n(x)$ and $v_n := F^n(v)$.
    Then $\|v_{n+1}\| < \lam(x_n) \|v_n\|$,
    so that $\|v_n\| < 2^{-n} \delta$ for all $n  \ge  0$.
    As $2^{-n} \delta < \eta$ for large $n$,
    there is $k  \ge  0$ and a vector $w \in \Ecu(x_k)$
    such that $\exp_{x_k}(w)$ and $\exp_{x_k}(v_k)$ lie on the same stable
    leaf and
    $F^n(w) - v_{n+k} < \delta/2$ for all $n \ge 0$.
    Without loss of generality, assume $k=0$ and write $w_n = F^n(w)$.
    If $w = 0$, the result is proved.
    Therefore, we assume $w  \ne  0$.
    Then $w  \in  \Cstar$ implies $w_n  \in  \Cstar$
    and therefore $\| w_{n+1} \| > 2 \lam(x_n) \|w_n\|$ for all $n  \ge  0$.
    However, one can show that
    $\| v_{n+1} - w_{n+1} \| < \lam(x_n) \|v_n - w_n\|$
    for all large $n$, and this gives a contradiction.
\end{proof}
\begin{notation}
    For the rest of the section,
    if $x_0$ and $y_0$ are distinct points on $\Lam$
    define $x_n = f^n(x_0)$ and $y_n = f^n(y_0)$ for all $n \in \bbZ$.
    For those indices where $d(x_n, y_n) < 1$,
    define $v_n \in T_{x_n}M$ such that $\|v_n\|<1$ and $y_n = \exp_{x_n}(v_n)$.
\end{notation}
\begin{lemma} \label{lemma:uniformcone}
    There is a uniform constant $N > 0$ such that
    for any distinct
    $x_0,y_0 \in \Lam$
    either $d(x_n, y_n) > \delta$ or $v_n  \in  \Cstar$
    for some $0  \le  n < N$.
\end{lemma}
\begin{proof}
    First note that since $\Ws(x_0) \cap \Lam = \{x_0\}$ by assumption,
    \cref{lemma:Fstable} implies that such an $N$ exists for each pair
    $(x_0, y_0)$
    considered on its own.
    The goal is to find a uniform constant $N$ which works for all pairs.
    Let $0 < \ep < \delta$ be such that $\|F(v)\| < \ep$
    implies $\|v\| < \delta$.
    The set
    \begin{math}
        \{(x,y) \in \Lam \times \Lam : \ep  \le  d(x,y)  \le  \delta\}
    \end{math}
    is compact.
    One may then use an open cover to show that
    there is a uniform constant $N > 0$ such that
    if $\ep  \le  d(x_0, y_0)  \le  \delta$ then
    either $d(x_n, y_n) > \delta$ or $v_n  \in  \Cstar$
    for some $0  \le  n < N$.

    Now suppose $0 < d(x_0, y_0) < \ep$.
    Let $m  \ge  0$ be the smallest integer such that
    either $\|v_{-m}\| \ge \ep$ or $v_{-m}  \in  \Cstar$.
    Such an $m$ must exist as $F \inv$ uniformly expands
    vectors in $\Cone$.
    If $v_{-m}  \in  \Cstar$, then $v_{-m+1}  \in  \Cstar$
    and the minimality of $m$ implies that $m = 0$.
    If $\|v_{-m}\| \ge \ep$, then $\|v_{-m}\|<\delta$ by the choice of $\ep$
    and so there is $0  \le  n < N$ such that $v_{n-m}  \in  \Cstar$.
    Since $m  \ge  0$,
    this implies that $v_n  \in  \Cstar$.
\end{proof}
\begin{cor}
    There is a sequence $\{\ep_n\}$ of positive numbers
    such that if $d(x_0, y_0) < \ep_n$ then
    $v_0  \in  Df^n(\Cstar)$.
\end{cor}
\begin{proof}
    Using $\delta$, $r_n$, and $N$ as above,
    take $\ep_n > 0$ small enough that
    \[
        d(x_0,y_0) < \ep_n  \quad \Rightarrow \quad  d(x_{-k}, y_{-k}) < \min\{\delta, r_n\}
    \]
    for all $0  \le  k  \le  N + n + 1$.
    By \cref{lemma:uniformcone}, $v_{-k}  \in  \Cstar$ for some $k  \ge  n+1$
    which further implies that $v_{-n-1}  \in  \Cstar$.
    The result then follows
    from \cref{lemma:coneback}.
\end{proof}
Since 
\begin{math}
    \bigcap_{n  \ge  0} Df^n(\Cstar) = \Ecu,
\end{math}
this shows that $\Lam$ is a $C^1$ submanifold tangent to $\Ecu$.

\section{Cross sections} \label{sec:cross} 

To prove \cref{thm:phcross},
we will combine \cref{thm:toperiodic} with the following result,
applied to a flow along the unstable direction.

\begin{thm} \label{thm:cross}
    Let $M$ be a compact connected manifold with boundary,
    $\psi$ a $C^0$ flow on $M$,
    $g:M \to [0,1]$ a continuous function,
    and $\ell > 0$ a constant
    such that
    \begin{enumerate}
        \item $g(\del M) = \{0, 1\},$ and

        \item if $x \in M$, $t \in \bbR$, 
        $0 < g(x) < 1$
        and $|t| > \ell$, then $g(\psi_t(x)) \in \{0,1\}$.
    \end{enumerate}
    Then, there is a compact codimension one submanifold $S$
    in the interior of $M$ which intersects any orbit in at most one point.
\end{thm}
Assume now that the hypotheses in the \cref{thm:cross} hold.
Note that $\psi$ is a global flow defined for all time.
For each $t \in \bbR$, $\psi_t:M \to M$ is a homeomorphism,
and so $\del M$ is invariant under the flow.
For $i,j \in \{0,1\}$,
define
\[
    X_{i,j} = \big\{x \in M : \lim_{t \to -\infty} g(\psi_t(x)) = i
                         \,\,\text{and}\,
                         \lim_{t \to +\infty} g(\psi_t(x)) = j \big\}.
\]
Since $g(\del M) = \{0,1\}$,
at least one boundary component is contained in $X_{0,0}$
and at least one is contained in $X_{1,1}$.
The second item in the theorem
implies that $M = X_{0,0} \cup X_{0,1} \cup X_{1,0} \cup X_{1,1}$.

\begin{lemma}
    The subsets $X_{0,0}$ and $X_{1,1}$ are closed.
\end{lemma}
\begin{proof}
    Suppose $\{x_k\}$ is a sequence in $X_{0,0}$,
    converging to $x \in M \without X_{0,0}$.
    Then there is $s \in \bbR$ such that
    $g(\psi_s(x))  \ne  0$.
    As $\psi_s$ is continuous,
    $g(\psi_s(x_k))  \ne  0$ for all large $k$.
    Then $g(\psi_t(x_k)) = 0$ for all large $k$ and all $t \in \bbR$ with
    $|t-s|>\ell$.
    By continuity, $g(\psi_t(x)) = 0$
    for all $t$ with $|t-s|>\ell$.
      \end{proof}
\begin{cor}
    At least one of $X_{0,1}$ or $X_{1,0}$ is non-empty.
\end{cor}
\begin{proof}
    Otherwise, $X_{0,0}$ and $X_{1,1}$ disconnect $M$ into two clopen subsets.
\end{proof}
Without loss of generality, assume $X_{0,1}$ is non-empty.

Also define $U_{0,0}$ as
\[
    U_{0,0} = \{x \in X_{0,0} :
    \, \text{there is $t \in \bbR$ such that $g(\psi_t(x))  \ne  0$} \}
\]
and define $U_{1,1}$ similarly.

\begin{lemma} \label{lemma:XUopen}
    The subsets $X_{0,1}$, $X_{1,0}$, $U_{0,0}$, and $U_{1,1}$ are open.
\end{lemma}
\begin{proof}
    If $x \in X_{0,1}$, then there is $t \in \bbR$
    such that
    \[
        g(\psi_{t-\ell}(x))
        < \tfrac{1}{3} <
        g(\psi_{t}(x))
        < \tfrac{2}{3} <
        g(\psi_{t+\ell}(x)).
    \]
    This inequality also holds for all points in a neighbourhood $V$ of $x$
    and implies that $V \subset X_{0,1}$.
    If $x \in U_{0,0}$, then there is $t \in \bbR$ and $\delta > 0$ such that
    \[
        g(\psi_{t-\ell}(x))
        < \delta <
        g(\psi_{t}(x))
        > \delta >
        g(\psi_{t+\ell}(x)).
    \]
    This also holds on a neighbourhood of $x$ and shows that
    $U_{0,0}$ is open.
    The cases of $X_{1,0}$ and $U_{1,1}$ are analogous.
\end{proof}
For the remainder of the proof, we assume $\ell < 1$.
This can always be achieved by rescaling time for the flow,
and makes the definitions simpler in what follows.
We now adapt the averaging method of Fuller \cite{fuller1965}
to this setting.
For each integer $n  \ge  1$, define $g_n: M \to [0,1]$ by
\[
    g_n(x) = \tfrac{1}{2n} \int_{-n}^{+n} g(\psi_t(x))\, dt.
\]
Let $\Dpsi$ denote the derivative along the flow.
That is, for a function $\alpha:M \to \bbR$, define
\[
    \Dpsi \alpha(x) := \lim_{t \to 0} \tfrac{1}{t}
    \big( \alpha(\psi_t(x)) - \alpha(x) \big).
\]
The Fundamental Theorem of Calculus implies that
\[
    \Dpsi g_n(x) = \tfrac{1}{2n}
    \big[ g(\psi_{n}(x)) - g(\psi_{-n}(x)) \big].
\]
The assumption $\ell < 1$ implies that if $x \in X_{0,1}$,
then at least one of
$g(\psi_{-n}(x)) = 0$
or
$g(\psi_{n}(x)) = 1$
holds.
Hence,
\begin{math}
    0  \le  \Dpsi g_n(x)  \le  \tfrac{1}{2n}
\end{math}
for all $n$ and
\begin{math}
    \Dpsi g_n(x) = \tfrac{1}{2n}
\end{math}
for a fixed $x  \in  X_{0,1}$ and large $n$.
Define $p:M \to [0,1]$ by
\[
    p(x) = \sum_{n=1}^\infty 2^{-n} g_n(x).
\]
If $x \in X_{0,1}$, one can show that
$\lim_{t \to -\infty} p (\psi_t(x)) = 0$,\,
$\lim_{t \to +\infty} p (\psi_t(x)) = 1$, and
\[
    \Dpsi p(x) = \sum_{n=1}^{\infty} 2^{-n} \Dpsi g_n(x) > 0.
\]
Hence, 
any orbit in $X_{0,1}$ intersects
$p \inv(\tfrac{1}{2})$
in exactly one point.
As in \cite{fuller1965},
one can then show that locally $X_{0,1} \cap \phalf$
has the structure of a codimension one $C^0$ submanifold.


\begin{lemma} \label{lemma:xpopen}
    For $i,j \in \{0,1\}$,
    the subset $X_{i,j} \cap \phalf$ is open in the topology of
    $\phalf$.
\end{lemma}
\begin{proof}
    The cases of $X_{0,1}$ and $X_{1,0}$ follow immediately from
    \cref{lemma:XUopen}.
    If $x \in X_{0,0} \cap \phalf$,
    then $p(x) = \half$ implies that $g(\psi_t(x))$
    cannot be zero for all $t$.
    Thus,
    $X_{0,0} \cap \phalf = U_{0,0} \cap \phalf$
    and is therefore open.
    Similarly for $X_{1,1} \cap \phalf$.
\end{proof}
\begin{cor} \label{cor:Xmanifolds}
    The set $X_{0,1} \cap \phalf$ is a finite disjoint union of compact connected
    codimension one $C^0$ submanifolds.
\end{cor}
\begin{proof}
    As noted above,
    $X_{0,1} \cap \phalf$ locally has the structure of a $C^0$ submanifold.
    By \cref{lemma:xpopen} and the fact that $M$
    splits into the disjoint union
    $M = X_{0,0} \cup X_{0,1} \cup X_{1,0} \cup X_{1,1}$,
    the subset $X_{0,1} \cap \phalf$
    is clopen in the topology of $\phalf$.
    In particular, this subset is compact and therefore consists of a finite
    number of compact connected submanifolds.
\end{proof}
To complete the proof of \cref{thm:cross},
take $S$ to be one of the components of $X_{0,1} \cap \phalf$.

\medskip

We now look at how these components interact with a diffeomorphism which
preserves the orbits of the flow.
In what follows,
let $S_1, \ldots, S_m$ be the connected components of $X_{0,1} \cap \phalf$.

\begin{prop} \label{prop:crossinvt}
    Suppose $f:M \to M$ is a homeomorphism which preserves the orbits of $\psi$
    and such that
    \[
            \lim_{t \to +\infty} g \psi_t(x) = 
            \lim_{t \to +\infty} g f \psi_t(x)
            \qandq
            \lim_{t \to -\infty} g \psi_t(x) = 
            \lim_{t \to -\infty} g f \psi_t(x)
    \]
    for all $x  \in  M_0$.
    Then there is $k  \ge  1$ such that an orbit of $\psi$ intersects
    a component $S_i$ if and only if the orbit intersects $f^k(S_i)$.
\end{prop}
\begin{proof}
    The hypotheses imply that $f(X_{0,1}) = X_{0,1}$.
    For each point $x \in f(S_i)$, there is a unique point $\hat x$
    on the orbit of $x$
    such that $p(\hat x) = \half$.
    Moreover, $\hat x$ depends continuously on $x$.
    The image $\{ \hat x : x \in f(S_i) \}$
    is a compact manifold in $X_{0,1} \cap \phalf$
    and is therefore equal to one of the $S_j$.
    This shows that, up to flowing along the orbits of $\psi$,
    the homeomorphism $f$ permutes the components $S_i$.
    Up to replacing $f$ again by an iterate,
    we may assume this is the identity permutation.
\end{proof}
\begin{proof}
    [Proof of \cref{thm:phcross}]
    First, consider the case where $\Eu$ is orientable.
    Define a $C^0$ flow $\psi$ such that the orbits of $\psi$
    are exactly the unstable leaves of $f$.
    This flow satisfies the hypotheses of \cref{thm:cross}
    (with $M_0$ in place of $M$).
    Consequently, there is a compact $C^0$ submanifold $S$
    in the interior of $M_0$
    which intersects each unstable leaf in at most one point.
    By \cref{prop:crossinvt},
    there is an iterate $f^k$ such that an unstable leaf intersects
    $f^k(S)$ if and only if it intersects $S$.
    Then, $\dist_u(S, f^k(S)) < \infty$ and \cref{thm:toperiodic}
    implies that there is a compact periodic center-stable
    submanifold $\Lam$ as desired.
    This concludes the orientable case.

    Instead of handling the non-orientable case directly,
    we assume now that there is an involution $\tau:M \to M$
    which commutes with $f$, 
    preserves the unstable foliation, 
    and reverses the orientation of $\Eu$.
    If $\Wu(\Lam)$ and $\Wu(\tau(\Lam)$) are disjoint,
    then $\Lam$ and $\tau(\Lam)$ are disjoint.
    If, instead, $\Wu(\Lam)$ intersects $\Wu(\tau(\Lam)$),
    then the argument
    in proof of \cref{prop:crossinvt}
    shows that $\dist_u(\Lam, \tau(\Lam)) < \infty$
    and the fact that $\Lam$ is $f$-periodic implies that $\Lam = \tau(\Lam)$.
    In either case,
    $\Lam$ quotients down to a compact submanifold embedded in $M/\tau$.
\end{proof}
\section{Making a calzone} \label{sec:calzone} 

In this section, we construct the example in \cref{thm:calzone}.
As in \cref{sec:fromzero},
it is slightly easier from the notational viewpoint
to switch the roles of $\Es$ and $\Eu$
in the construction.
Therefore, we will actually build a system with a normally repelling fixed
point and two intersecting center-unstable tori.

First, we build a partially hyperbolic subset of
$\bbT^2 \ti \bbR$
which is the union of two non-disjoint $cu$-tori.
Then, we explain how this partially hyperbolic subset
can be glued into the 3-torus in such a way to produce
a global dominated splitting.

The two $cu$-tori each have the same derived-from-Anosov dynamics
with a repelling fixed point.  They are glued together on the complement of
the basin of repulsion of this fixed point.
The $cu$-tori are, of course, tangent along this intersection
and the construction vaguely resembles the type of food called a calzone,
where two pieces of dough are pressed together to enclose a region
which is full of other ingredients.
A depiction of this construction is given in figure \ref{fig:calzone}.
\begin{figure}
    \centering
    \includegraphics{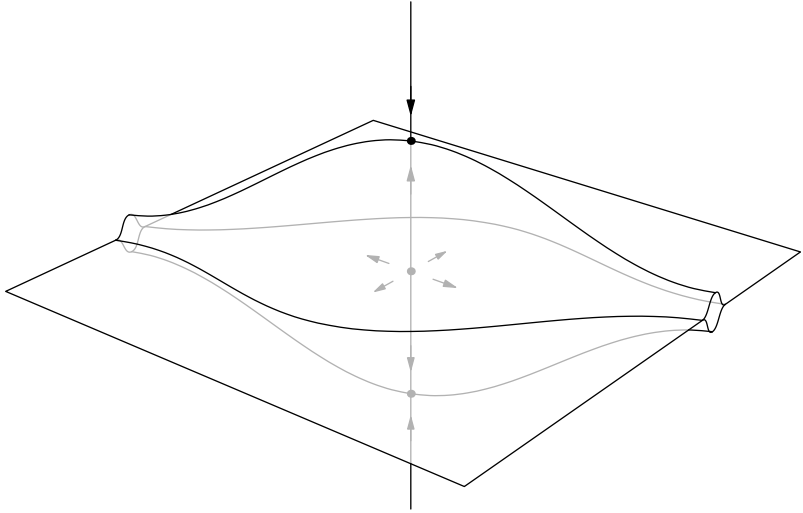}
    \caption[Calzone]{ Two intersecting center-unstable tori.}
    \label{fig:calzone}
\end{figure}
\medskip

Let $g : \bbT^2 \to \bbT^2$ be a weakly partially hyperbolic diffeomorphism
with a splitting of the form $\Eu \oplus \Ec$.
That is, 
\[
    \|Df v^c\| < \|Df v^u\|
    \qandq
    1 < \|Df v^u\|
      \]
hold for all $x  \in  M$ and unit vectors
$v^c  \in  \Ec(x)$, and
$v^u  \in  \Eu(x)$.
Further suppose that $q  \in  \bbT^2$ is a repelling fixed point for $g$.
Let
\[
    B(q) = \big\{ \, x  \in  \bbT^2 : \lim_{n \to \infty} g \invn(x) = q \, \big\}
      \]
be the basin of repulsion of $q$
and define $K = \bbT^2 \sans B(q)$.
Define a constant $0 < \lam < 1$ small enough that
\begin{math}
    \|Dg \, v\| > 2 \lam
\end{math}
for all unit vectors $v  \in  \bbT^2$.
Define a smooth function $\bt : \bbR \to \bbR$
with the following properties:{}
\begin{enumerate}
    \item $\bt$ is an odd function
    with fixed points exactly at -1, 0, and +1;

    \item the fixed point at zero is expanding with
    $1 < \bt'(0) < \|Dg v\|$
    for any unit vector $v  \in  T_q \bbT^2$;

    \item the fixed points -1 and +1 are contracting with
    $\bt'(-1) = \bt'(+1) = \lam$; and

    \item there is a constant $C > 1$ such that
    $\bt(s) = \lam s$ for all $s  \in  \bbR$ with $|s| > C$.
\end{enumerate}
\begin{figure}
    \centering
    \includegraphics{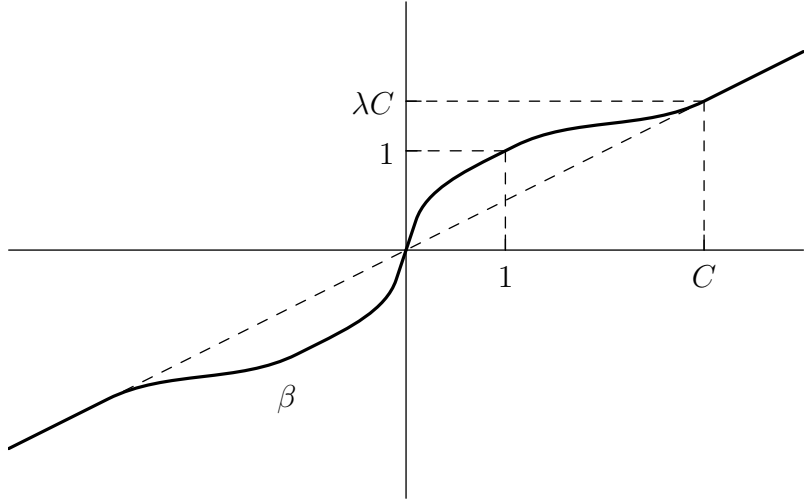}
    \caption[Beta]{The graph of the function $\bt$.}
    \label{fig:beta}
\end{figure}
See figure \ref{fig:beta}.
Define a smooth function $\alpha: \bbT^2 \to [0,1]$
which 
equals $0$ on a neighbourhood of $K$
and
equals $1$ on a neighbourhood of $q$.
Then define $f:\bbT^2 \times \bbR \to \bbT^2 \times \bbR$
by
\[
    f(x,s) = \big( g(x), \,
        (1 - \al(x)) \lam s  + \al(x)\bt(s) \big).
\]
We now look at the behaviour of tangent vectors under the
action of the derivative.
If $p = (x,s)  \in  \bbT^2 \ti \bbR$,
a tangent vector $u  \in  T_p (\bbT^2 \ti \bbR)$ may be decomposed
as $u = (v,w)$
with horizontal component $v  \in  T_x \bbT^2$ and
vertical component $w  \in  T_s \bbR$.

\begin{lemma} \label{lemma:vfate}
    For any point $p = (x,s)  \in  \bbT^2 \ti \bbR$ and
    any tangent vector
    \[    
        u = (v,w)  \in  T_p (\bbT^2 \ti \bbR),
    \]
    define $u_n = Df^n(u)$
    and let $v_n$ and $w_n$ be its horizontal and vertical components
    respectively.
    \begin{enumerate}
        \item If $v$ is non-zero, then the ratio
        \begin{math}
            \frac{\|w_n\|}{\|v_n\|}  \end{math}
        tends to 0 as $n  \to  +\infty$.

        \item If $v \notin \Ec_g(x)$, then the angle between
        $v_n$ and $\Eu_g(g^n(x))$ tends to 0
        as $n  \to  +\infty$.
    \end{enumerate}  \end{lemma}
\begin{proof}
    The non-wandering set of $f$ is
    \[
        NW(f) = \big( K \ti \{0\} \big)
        \cup
        \big( \{q\} \ti \{-1,0,+1\} \big).
    \]
    At $(q,0)$ the condition on $\bt'(0)$
    implies that vectors in the horizontal direction are
    expanded more strongly
    than vectors in the vertical direction.
    At all other points in $NW(f)$,
    the condition on $\lam$ implies that
    a vector in the vertical direction is contracted more strongly
    than any vector in the horizontal direction.
    Hence, there is a neighbourhood $U$ of $NW(f)$ and a constant $\sig < 1$
    such that if $Df^n(p)  \in  U$,
    then
    \[
        \frac{\|w_{n+1}\|}{\|v_{n+1}\|}
         \le  \sig
        \frac{\|w_n\|}{\|v_n\|}.
    \]
    Since $Df^n(p)  \in  U$ for all large $n$, this implies item (1).

    From the definition of $f$,
    note that $v_n = Dg^n(v)$ for all $n$,
    and item (2) follows directly from fact that $g$ is weakly partially
    hyperbolic.
\end{proof}
From \cref{lemma:vfate},
one may show that on the invariant subset
\[
    X := \bigcap_{n  \ge  0} f^n \big( \bbT^2 \ti [-C,C] \big)
\]
there is a dominated splitting with three one-dimensional subbundles.
We will use $\Eu \oplus \Ec \oplus \Es$ to denote this splitting,
even though the $\Es$ direction is not uniformly contracting.

The fixed point $(q,1)$
is hyperbolic with a two-dimensional unstable direction.
Let $\Wcu(q,1)$ denote the two-dimensional unstable manifold
though this point.
This manifold may be expressed as the graph of a $C^1$ function
from $B(q)$ to $\bbR$.
Let $T_+$ be the closure of $\Wcu(q,1)$.
Then $T_+$ may be expressed as the graph of a continuous function
from $\bbT^2$ to $\bbR$ which is zero on all points in $K$.
One may show,
either directly or by a variant of \cref{thm:fromzero},
that $T_+$ is a $C^1$ submanifold tangent to $\Ec \oplus \Eu$.
Since $\Es$ is uniformly attracting on $T_+$,
this implies that the tangent bundle restricted to $T_+$
has a strongly partially hyperbolic splitting.
By symmetry, the closure of the unstable manifold through
the point $(q,-1)$
is also a surface, denoted $T_-$, with similar properties.
Thus, the union $T_+ \cup T_-$
is a partially hyperbolic set
and the intersecton
$T_+ \cap T_- = K \ti 0$ is non-empty.

\medskip

We now describe how this example may be embedded into $\bbT^3$.
The constant $C > 1$ was defined so that the equality
\begin{math}
    f(x,s) = (g(x), \lam s)
\end{math}
holds for all $(x,s)$ with $|s| > C$.
By rescaling the vertical $\bbR$ direction of $\bbT^2 \ti \bbR$,
one may, for any given $\ep > 0$,
define a similar example such that this equality
holds for all $(x,s)$ with $|s| > \tfrac{\ep}{4}$.
Then, take the construction of $f$ given in the proof of
\cite[Theorem 1.2]{ham20XXconstructing}
and replace the dynamics on $\bbT^2 \ti [-\tfrac{\ep}{2}, \tfrac{\ep}{2}]$
defined there
with that of the $f$ defined here.
Using \cref{lemma:vfate} and the techniques in \cite{ham20XXconstructing}
one may show that this new system has a global dominated splitting
and that, outside a basin of repulsion,
this dominated splitting is partially hyperbolic.
This establishes all of the properties listed in \cref{thm:calzone}.

As a final note, it is possible to define a variation on this example
by composing $f$ with the reflection $(x,s) \mapsto (x,-s)$.
This new system has two compact center-unstable tori which 
intersect and which are the images
of each other.

\section{No calzones} \label{sec:nocalzone} 

The last section constructed an example which was only partially hyperbolic
on a subset of $\bbT^3$.
Here we prove \cref{thm:nocalzones}, showing that the example cannot be improved
to a global partially hyperbolic splitting.
The basic idea of the proof is that the region between the two tori must have
finite volume, even after lifting to the universal cover.
This region also has unstable curves of infinite length.
The ``length-versus-volume'' argument of \cite{BBI2} then gives a contradiction.

Assume $f$ is a partially hyperbolic diffeomorphism of a
3-manifold $M$,
and that $T_0$ and $T_1$ are two intersecting
compact $cs$-submanifolds.
Up to replacing $f$ by an iterate, assume each $T_i$ is $f$-invariant.
Up to replacing $M$ by a double cover, assume $M$ is orientable.
The results in \cite{rhrhu2011tori} then imply
that $M$ is either
\begin{enumerate}
    \item the 3-torus,

    \item the suspension of ``minus the identity'' on $\bbT^2$, or
    
    \item the suspension of a hyperbolic toral automorphism on $\bbT^2$.
\end{enumerate}
We only consider the case $M = \bbT^3 = \bbR^3 / \bbZ^3.$
The other two cases have analogous proofs.
Further, after applying a $C^1$ change of coordinates to the system,
we assume without loss of generality that $T_0 = \bbT^2 \ti 0$.

The lifted partially hyperbolic map $f:\bbR^3 \to \bbR^3$
on the universal cover
is a finite distance from a map of the form $A \times \id$
where $A:\bbR^2 \to \bbR^2$ is linear and hyperbolic.
The subset $S_0 := \bbR^2 \ti 0$ covers $T_0$
and is invariant under the lifted dynamics.
By a slight abuse of notation,
if $x = (x_1,x_2,x_3)  \in  \bbR^3$ and $z = (z_1,z_2)  \in  \bbZ^2$,
then write
\[
    x + z = (x_1+z_1,x_2+z_2,x_3).
\]
Let $H:\bbR^3 \to \bbR^2$ be the Franks semiconjugacy \cite{Franks1}.
That is, $H$ is a uniformly continuous surjection such that
$H f(x) = A H(x)$ and $H(x+z) = H(x) + z$
for all $x  \in  \bbR^3$ and $z  \in  \bbZ^2$.

Up to replacing $f$ by an iterate, assume the eigenvalues of $A$
are positive.
Let $\lam > 1$ be the unstable eigenvalue.
There is a non-zero linear map $\piu : \bbR^2 \to \bbR$
such that $\piu(Av) = \lam \piu(v)$
for any $v  \in  \bbR^2$.
We will also consider $\piu$ as a map from $\bbR^3$ to $\bbR$
which depends only on the first two coordinates of $\bbR^3$.
Define $\Hu = \piu \circ H$.
One may then verify the following properties of $\Hu$
hold for any $x,y  \in  \bbR^3$ and $z  \in  \bbZ^2$.
\begin{enumerate}
    \item $\Hu(f(x)) = \lam \Hu(x)$;

    \item $\Hu(x+z) = \Hu(x) + \piu(z)$;

    \item if $x$ and $y$ are on the same stable leaf of $f$, then $\Hu(x) = \Hu(y)$;
    and

    \item there is a uniform constant $R > 0$ such that $|\Hu(x) - \piu(x)| < R$.
\end{enumerate}

\begin{lemma} \label{lemma:Uconst}
    Suppose $U$ is a non-empty proper subset of $S_0$
    which is saturated by stable leaves and
    which is invariant under translations by $\bbZ^2$.
    Then $\Hu$ is constant on any connected component of $U$.
\end{lemma}
\begin{proof}
    Fix some
    non-zero element $z  \in  \bbZ^2$ with $\piu(z) > 0$ and
    consider any point $p  \in  S_0$.
    The set
    \[    
        S_0 \sans \big( \Ws(p) \cup \Ws(p+z) \big)
    \]
    consists of three connected components, $V_1$, $V_2$, $V_3$.
    Up to relabelling these components, one may show using items (3) and (4)
    of the list above that
    \[
        \big \{ x  \in  S_0 \: : \: \piu(x) < \piu(p) - R \big \} \subof V_1
    \]
    and
    \[
        \big \{ x  \in  S_0 \: : \: \piu(x) > \piu(p+z) + R \big \} \subof V_3.
    \]
    The remaining component $V_2$ is then
    is bounded in the sense that
    \[
        |\piu(x) - \piu(y)| < 2 R + \piu(z)
    \]
    for all $x,y  \in  V_2$.
    Note that this bound is independent of the choice of $p$.

    Choose a point $q$ in $S_0 \sans U$.
    Then $U$ is a subset of
    \[
        S_0 \sans \bigcup_{k  \in  \bbZ} \Ws(q + k z).
    \]
    If $x$ and $y$ are points on the same connected component
    $U_0$ of $U$, then
    one may find a point $p$ of the form  $p = q + k z$
    so that $V_2$ as defined above satisfies $U_0 \subof V_2$.
    Then, the above bound holds for $x$ and $y$.
    For any $n  \in  \bbZ$,
    the set $f^n(U)$ also satisfies the hypotheses of this lemma,
    and so
    \[
        \Big|\piu f^n(x) - \piu f^n(y) \Big| \, < \, 2 R + \piu(z).
    \]
    Then
    \[
        \lam^n \Big|\Hu(x) - \Hu(y)\Big| \, = \,
        \Big|\Hu f^n(x) - \Hu f^n(y) \Big| \, < \, 4 R + \piu(z)
    \]
    for all $n  \ge  0$, and this shows that $|\Hu(x) - \Hu(y)| = 0$.
\end{proof}

Since the tori $T_0$ and $T_1$ intersect,
we may 
lift $T_1$ to a surface $S_1$ which intersects $S_0$.
Since $S_1$ is invariant under translation by $\bbZ^2$,
there is $N > 0$ such that $S_1 \subset \bbR^2 \times [-N, N]$.

Note that $S_0 \sans S_1$ is saturated by stable leaves
and is invariant under translations by $\bbZ^2$.
Let $X_0$ be the closure of a connected component of $S_0 \without S_1$.
Using \cref{lemma:Uconst},
one can show that $\Hu$ is constant on $X_0$.
There is a corresponding set $X_1$ which the closure of
a connected component of $S_1 \without S_0$ and such that the intersection
$X_0 \cap X_1$ consists of two stable manifolds.
The union $X := X_0 \cup X_1$ may be regarded as two infinitely long
strips glued together along their boundaries.
Hence, $X$ is a properly embedded topological cylinder cutting
$\bbR^3$ into two pieces.

By adapting the proof of \cref{lemma:Uconst},
one may show that $\Hu$ is constant on $X_1$.
Then $\Hu$ is constant on all of $X$ and $\piu(X)$ is a bounded subset of $\bbR$.
Thus, one of the connected components of $\bbR^3 \without X$
is a region $Y$ with $\piu(Y)$ bounded.
Note that $\Hu$ is constant on each $f^n(\del Y)$ for $n \in \bbZ$
and therefore the length of $\piu f^n(Y)$ is uniformly bounded for all $n$.

Similarly to $\piu$, define a linear map $\pis$
such that the composition $\Hs = \pis \circ H$
satisfies $\Hs(f(x)) = \lam \inv \Hs(x)$
and $|\Hs(x) - \pis(x)| < R$
for all $x  \in  \bbR^3$.
Let $J$ be an unstable segment inside $Y$.
Then, the lengths of both $\pis f^n(J)$ and $\piu f^n(J)$
are uniformly bounded for all $n  \ge  0$.
Further, $f^n(J) \subset \bbR^2 \times [-N,N]$.
Thus, the diameters of $f^n(J)$ are uniformly bounded for all $n  \ge  0$.
The results in \cite{BBI2} imply that since the lengths of the
unstable curves $f^n(J)$ grow without bound,
their diameters as subsets of $\bbR^3$ must grow without bound as well.
This gives a contradiction.


\acknowledgement
The author thanks Rafael Potrie for helpful conversations.


\bibliographystyle{alpha}
\bibliography{dynamics}

\end{document}